\documentclass[twoside]{amsart} 

\usepackage{amsmath, amsfonts, amsthm, amssymb}
\usepackage{exscale}
\usepackage{cite}
\usepackage{amscd}
\usepackage{graphics}
\usepackage{color}
\usepackage{dsfont}
\usepackage{graphicx}

\renewcommand{\geq}{\geqslant}
\renewcommand{\leq}{\leqslant}

\newtheorem{theorem}{Theorem}

\newtheorem{proposition}{Proposition}[section]

\newtheorem{lemma}[proposition]{Lemma}
\newtheorem*{main-theorem}{Main Theorem}
\newtheorem*{theorem*}{Theorem}

\theoremstyle{definition}

\newtheorem{remark}[proposition]{Remark}

\newtheorem*{remark*}{Remark}

\numberwithin{equation}{section}

\def\phi{\varphi}

\def\ZZ{{\mathbb Z}}

\def\reals{{\mathbb R}}
\def\cx{{\mathbb C}}

\def\Ci{{\mathcal C}^\infty}
\def\Re{\,\mathrm{Re}\,}
\def\Im{\,\mathrm{Im}\,}

\def\div{\mathrm{div}\,}
\def\supp{\mathrm{supp}\,}
\def\Arg{\mathrm{Arg}\,}
\def\id{\,\mathrm{id}\,}

\def\O{{\mathcal O}}

\def\Op{\mathrm{Op}\,}

\def\phi{\varphi}

\def\be{\begin{eqnarray*}}
\def\ee{\end{eqnarray*}}
\def\ben{\begin{eqnarray}}
\def\een{\end{eqnarray}}

\def\lll{\left\langle}
\def\rrr{\right\rangle}

\def\L2R{L_{\text{Rest}}^2}

\def\11{\mathds{1}}

\def\tchi{\tilde{\chi}}
\def\L2c{L^2_{\text{comp}}}

\def\tV{\widetilde{V}}

\def\tP{\widetilde{P}}

\def\tgamma{\tilde{\gamma}}

\def\tR{\tilde{R}}

\def\Hsc{H_{\text{sc}}}
\def\tg{\tilde{g}}
\def\dvol{d\text{vol}_g}
\def\tOmega{\tilde{\Omega}}

\def\p{\partial}

\def\tI{\tilde{I}}
\def\trho{\tilde{\rho}}
\def\tK{\tilde{K}}

\begin{document}

\title[Damped Wave Equation]{Imperfect geometric control and
  overdamping for the damped wave equation}

\author{Nicolas~Burq}

\author{Hans~Christianson}
\email{hans@math.unc.edu}
\address{Department of Mathematics, UNC-Chapel Hill \\ CB\#3250
  Phillips Hall \\ Chapel Hill, NC 27599}

\thanks{N.B. was supported in part by Agence Nationale de la Recherche
  project NOSEVOL, 2011 BS01019 01, and H.C. was supported in part by NSF grant DMS-1059618.}


\begin{abstract}
We consider the damped wave equation on a manifold with imperfect
geometric control.  We show the sub-exponential energy decay estimate in
\cite{Chr-NC-erratum} is optimal in the case of one hyperbolic
periodic geodesic.  We show if the equation is overdamped, then the
energy decays exponentially.  Finally we show if the equation is
overdamped but geometric control fails for one hyperbolic periodic
geodesic, then nevertheless the energy decays exponentially.

\end{abstract}

\maketitle

\section{Introduction}
\label{S:intro}

In this paper, we discuss several damped wave type problems in various
geometric settings in which the support of the damping term fails to
have perfect geometric control over the whole domain.  It is known
that some loss in regularity must occur to obtain energy decay,
however the rate of decay, as a function of time, is still an
important object to study.  The starting
point for our work is the example of Lebeau \cite{Leb-dw} and the mistake in the work of the second author
\cite{Chr-NC} (which has been corrected in \cite{Chr-NC-erratum}).
In the example of Lebeau \cite{Leb-dw}, the
stable/unstable manifolds of one hyperbolic orbit are homoclinic to
those of 
other hyperbolic orbits which are contained in the damping region, so
exponential energy decay still occurs.  In this paper, we analyze
the damped wave equation on a ``lumpy torus'' manifold, which has
similar characteristics to  the example of Colin de
Verdi\`ere-Parisse \cite{CdVP-II}, in which the stable/unstable manifolds of
of a hyperbolic periodic orbit are homoclinic to each other, and hence
``come back'' from the damping regions to the undamped region (this
phenomena also occurs in a ``double-well'' potential problem \cite{HeSj}).  In
this example, we prove the strongest rate of energy decay is
sub-exponential, which is the corrected statement in
\cite{Chr-NC-erratum}.  

Motivated by the viscous damping discussion in \cite{ErZu-damp}, we discuss also more general cases when geodesics may return from the
damping region, but with a stronger damping term.  In this case, we
prove exponential energy decay with a loss in derivative.


The techniques of proofs combine microlocal resolvent estimates near
the ``trapped set'', a gluing argument, and analysis of semiclassical
defect measures to estimate the asymptotic distribution of eigenvalues
for the damped wave operator.  The size of the neighbourhood between
the real axis and the spectrum gives the rate of decay, while the
estimates of the inverse in this neighbourhood give the loss in
derivatives (see, for example \cite{Bur-wd} and the adaptation in
\cite{Chr-wave-2}).

\subsection{Organization}
This note is organized as follows.  In \S \ref{S:CdV-P-example} we
look at a particular example (following Colin de Verdi\`ere-Parisse
\cite{CdVP-I,CdVP-II}) which shows the corrected estimate in
\cite{Chr-NC-erratum} is in general sharp, and that the example in \cite{Leb-dw} is a 
special case where this estimate can be improved.  In \S
\ref{S:overdamping-I}, we show that by adding a stronger damping term,
under the usual perfect geometric control assumption, we get an 
exponential energy decay similar to the weaker damped case (a similar
problem has been studied in \cite{ErZu-damp}).  In \S
\ref{S:overdamping-II}, we re-examine the ``black box'' type framework
from \cite{BuZw-bb} and \cite{Chr-NC,Chr-NC-erratum,Chr-QMNC,DaVa-gluing,CSVW} (which includes the example in \S
\ref{S:CdV-P-example}), and show that with the addition of a stronger
damping term, the corrected estimate from \cite{Chr-NC-erratum} can be
inproved. 

\section{Imperfect geometric control: an example}
\label{S:CdV-P-example}
In this section we study a particular example of a manifold with a
hyperbolic periodic geodesic and damping which controls the manifold
everywhere outside a neighbourhood of the geodesic.  Specifically, let
$M = \mathbb{T}^2$ be the $2$-dimensional periodic surface of revolution given by 
$$ M= \{ (x,y,z); x= R(z) \cos (\theta), y = R(z) \sin (\theta), z\in
\mathbb{T}= \mathbb{R}/ 2\pi \mathbb{Z}\}, 
$$
 equipped with the warped product 
metric 
(see Figure \ref{F:torus} for a schematic drawing)
\[
g = dz^2 + R^2(z) d\theta^2.
\]

We choose the function $R(z)$ to be even, have a minimum at $z = 0$, a
maximum at $z = \pi$ with no other critical points, and have a very specific structure near $z = 0$:
\[
R^{-2} (z) = 1-z^2,
\]
for $z$ in a small neighbourhood of $0$.  

\begin{figure}
\centering
\includegraphics[keepaspectratio,width=3in]{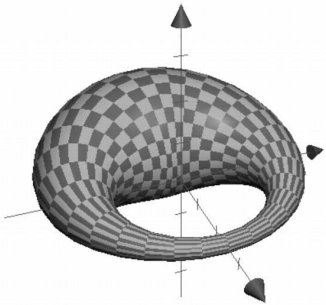}
\caption{The compact Riemannian manifold $(M,g)$.}\label{F:torus}
\end{figure}

The Riemannian volume element becomes 
\[
\dvol = R(z) dz d\theta
\]
and the Laplace-Beltrami operator is
\[
- \Delta_g = \frac {1} {R(z)}\partial_z  { R(z) } \partial _z + \frac 1 { R^2(z)} \partial_\theta^2.
\]
The manifold $M$ has a closed hyperbolic geodesic $\gamma$
characterized by $z = 0$ ($M$ also has a closed elliptic geodesic at
$z = \pm \pi$, but this section is concerned with the hyperbolic geodesic).
We consider the damped wave equation on $M$ under the assumption that
the damping term controls $M$ geometrically away from $\gamma$.
Let $a = a(z)$ be a smooth function of the $z$ variable alone
satisfying $a(z) \equiv 1$ away from $z = 0$ and $a(z) \equiv 0$ for
$z$ in a neighbourhood of $z = 0$.  
Assume further that $a$ is
symmetric about $z = 0$.  
We then consider the following equation on $M$:
\begin{equation}
\label{wave-equation-0}
\left\{ \begin{array}{l}
\left( \partial_t^2 - \Delta_g + a(z) \partial_t \right) u(z,\theta,t) = 0, \quad (z, \theta,t) \in M \times (0, \infty) \\
u(z, \theta,0) = u_0, \quad \partial_t u(z, \theta,0) = u_1.
\end{array} \right.
\end{equation}
\begin{theorem}\label{th.1}
Let $\delta >0$ and $E(t)$ be the energy for solutions to the damped wave equation
\eqref{wave-equation-0}. Assume that $f(t)$ is a function which satisfies
\[ \forall (u_0, u_1) \in H^{1+ \delta} \times H^{\delta}, 
E(t) \leq f(t)  \|(u_0,u_1)
\|_{H^{1+ \delta}\times H^{\delta}}^2 
\]
 Then there exists  $C, c_\delta>0$
such that 
\[
f(t) \geq C^{-1} e^{-c_\delta\sqrt{t}}.
\]
\end{theorem}
Formally cutting off for $t \leq 0$ and taking the Fourier transform
in time yields the following spectral equation:
\begin{align*}
P(\tau) \hat{u} := & ( -\tau^2 - \Delta_g +  i\tau a(z))\hat{u} \\
= & \hat{f},
\end{align*}
where $f$ is a function of the initial data $(u_0, u_1)$.  
In other words, understanding decay properties for solutions to the
damped wave equation is equivalent to understanding spectral
properties of the operator $P( \tau )$.  In particular, we want to
estimate the asymptotic distribution of the {\it imaginary} parts of
the $\tau_j$, where the $\tau_j$ are the eigenvalues of the
operator $P (
\tau)$.  Hence we consider the equation $P( \tau ) u = 0$. We now separate variables
\[
u(z, \theta) = \psi_{\tau,k}(z) e^{iky}, \,\, k \in \ZZ
\]
to get the following equation for $\psi_{\tau,k}$
\begin{equation}
\label{E:psi-sep-1}
\Bigl(-\frac {1} {R(z) }\partial_z  { R(z) } \partial _z + \frac {k^2} { R^2(z)} + i \tau a(z) - \tau^2 \Bigr) \psi_{\tau,k} = 0.
\end{equation}
We will ultimately be interested in the high-energy asymptotic regime
where $\Re \tau \sim |k| \to \infty$ and $|\Im \tau| \leq C$ for some $C>0$, which motivates writing
a semi-classical reduction $h= k^{-1}, \mu = h \tau$. We get 
$$ P^h_{\mu, a} = \Bigl(\frac {1} { R (z)} hD_z R(z) hD_z + \frac {1} { R^2(z)} + i h \mu a(z) - \mu^2 \Bigr)
$$
where $\mu$ ostensibly takes complex  values 
but we will be mainly interested in values of $\mu$ in a neighbourhood of
$1$. 

We further want to avoid any pesky issues with regards to the
Riemannian volume element, so we recall that if $T u = R^{1/2} u$,
then $T : L^2( R(z) dz ) \to L^2 (dz)$ is an isometry, and we can
conjugate our operator $P^h_{\mu,a}$ to get 
\[
T P^{h}_{\mu,a} T^{-1} = \Bigl((hD_z)^2 + \frac {1} { R^2(z)} + h^2
V_1(z) + i h \mu a(z) - \mu^2 \Bigr),
\]
which is an unbounded operator on $L^2(dz)$ with essentially self-adjoint
principal part.  The subpotential $V_1(z)$ involves derivatives of
$R(z)$, but in what follows, we are only interested in constructing
quasimodes of accuracy $\O(h^{2 - \delta} )$ for $\delta>0$, so the
$\O(h^2)$ subpotential is harmless.  Let us assume for the remainder
that we have conjugated and subtracted off the subpotential so that we
can concentrate on the important terms without getting bogged down in notation.

The semi-classical symbol of the operator $P^h_{\mu,a}$ is 
\begin{equation}
\begin{aligned}
& P_{\mu, a} ( z, \zeta, h)= p^0_{\mu, a} (z, \zeta) + h p^1_{\mu, a} (z, \zeta) \\
& p^0_{\mu, a} (z, \zeta)= \zeta ^2 + \frac 1 {R^2(z)}  - \mu^2 \\
& p^1_{\mu, a} (z, \zeta) = ia(z) \mu 
.
\end{aligned} \end{equation} 
\begin{theorem}\label{theo-quasi} 
  There exists sequences $h_n\rightarrow 0, \mu_n \rightarrow 1$ and $u_n\in L^2(M)$ such that 
\begin{itemize}
\item $ \Re (\mu_n) = 1+  \mathcal{O}(h_n)$
\item $\Im ( \mu_n ) =  \frac{ h_n} {\log (h_n^{-1})}$
\item $\|u_n \|_{L^2(M)} =1$
\item For any $\epsilon>0$, there exists $C>0$ such that for all $n \in \mathbb{N}$, 
$$ \|P^h_{\mu, a} (u_n) \|_{L^{2} ( M)} \leq C h_n^{2 - \epsilon}.$$ 
\end{itemize}
\end{theorem}
The idea to prove this result is basically to keep $\mu$ and $h$ as
parameters, keeping in mind that ultimately the two first properties
in Theorem~\ref{theo-quasi} will be satisfied, and construct
approximate solutions (quasi-modes) first on the outgoing and incoming
manifolds of the hyperbolic fixed point. Of course, the homoclinicity
of these manifolds implies by geometric optics constructions that
these quasi-modes on any point of each branch uniquely determine the
quasi modes everywhere on each branch. Then we apply the method
developed in~\cite{HeSj, CdVP-II, Leb-dw}, which shows that near the
hyperbolic fixed point, one can determine uniquely the quasi-modes in
the outgoing branches in terms of the quasi-modes on the incoming
branch, via a transfer operator. This strategy clearly leads us to an
overdetermined system: the quasi modes on the outgoing branch are
determined both by the geometric optics constructions and by the
transfer matrix procedure. To overcome this overdetemination, one has
to choose cleverly the parameters $h_n$ and $\tau_n$ (subject to some
Bohr-Sommerfeld type quantization rules). 
The existing literature on the subject of unstable critical points is
lacking in several places for our purposes.  The approach of Colin de Verdi\`ere-Parisse and
Helffer-Sj\"ostrand \cite{CdVP-II, HeSj} only applies to the
self-adjoint (real spectrum) setting, whereas the stationary damped
wave operator is manifestly nonself-adjoint.  
The approach of Lebeau
\cite{Leb-dw} allows for nonself-adjoint operators, but the $h$-Fourier
Integral Operators ($h$-FIOs) have an unfavorable dependence on the spectrum.
Hence, since we are only interested in an example situation anyway, we
choose our operator so that it is exactly quadratic near $(z, \zeta) = (0,0)$.
In this case, we can construct the $h$-FIO explicitly, independent of
the spectral parameter, and with no error term in the Egorov
transformation rule (see Lemma \ref{L:exact-nf} below).  This
simplifies our analysis significantly.


We write $\mu^2 = 1 + E + i F$ for $E$ small and real and $F = \O (h
)$ small and real.  
 The operator $P( \mu, a)$ has principal symbol 
 $$p^0_{\mu, a} (z, \zeta)=  \zeta ^2 + \frac 1 {R^2(z)} -1 -E$$ and the only critical elements of the Hamiltonian vector
field $H_p =    2 \zeta \partial_z + 2R'(z) R^{-3}(z) \partial_\zeta$ are at $(z, \zeta) =
(0,0)$ and $(z, \zeta) = (\pm \pi, 0)$.  We choose fix here  a metric
so that $\zeta^2 + R^{-2}(z) -1 = \zeta^2 - z^2$ near $z = 0$.
\begin{figure}
\includegraphics[keepaspectratio,width=3in]{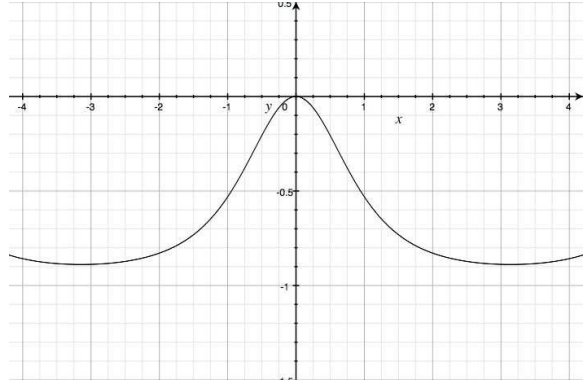}
\caption{\label{F:potential} The $1$ dimensional effective potential
  $R^{-2}(z) = (2-\cos z)^{-2} -1$.}
\end{figure}

Recalling the special exact quadratic structure of $\zeta^2 + R^{-2}(z)-1$
near the hyperbolic critical point $(0,0)$, Hamilton's ODEs become
\[
\begin{cases}
\dot{z} = 2 \zeta \\
\dot{\zeta} = 2z,
\end{cases}
\]
so that $z + \zeta = Ce^{2t}$ and $\zeta - z = C'e^{-2t}$.  This yields
the exact {\it local} phase portrait depicted in Figure \ref{F:phase1}.  The global
(periodic) 
phase portrait is depicted in Figure \ref{F:phase2}.  The fact we will
be using in this section is that the unstable manifolds near $(0,0)$
are homoclinic to the stable manifolds.  This is the opposite
situation to the example of Lebeau \cite{Leb-dw} in which the unstable
manifolds near the critical point at $(0,0)$ are {\it heteroclinic} to
the unstable manifolds near different critical points.
\begin{figure}
\input{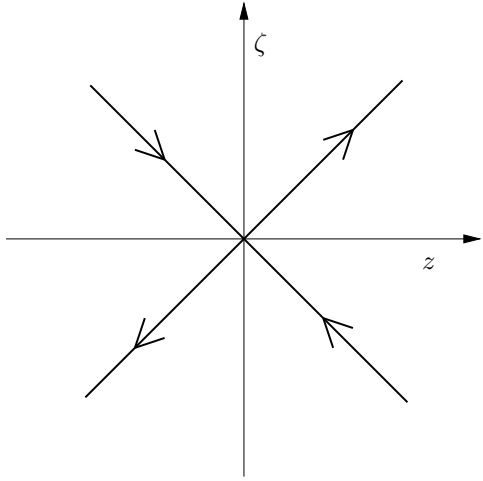}
\caption{\label{F:phase1} The local phase portrait near the hyperbolic
  fixed point $(0,0)$.}
\end{figure}

\begin{figure}
\input{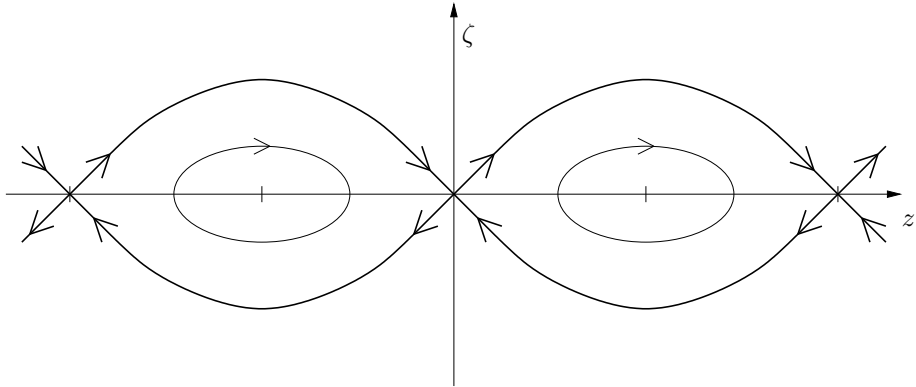}
\caption{\label{F:phase2} The global (periodic) phase portrait.  There
  is a hyperbolic fixed point at $(0,0)$ and an elliptic fixed point
  at $(\pm \pi, 0)$.  Observe that, owing to the periodicity, the
  unstable manifolds near $(0,0)$ is homoclinic to the stable
  manifolds.}
\end{figure}


\subsection{Microlocal constructions}
The starting point of the construction is to reduce the study to the model operator $x\partial_x$. This was already applied in similar contexts by Helffer-Sj\"ostrand~\cite{HeSj} and Colin de
Verdi\`ere-Parisse \cite{CdVP-II}. 

Since this is an example, we have chosen our function $R(z)$ to have a
nice structure near $z = 0$ so that a reduction to normal form is
simple and explicit.  For this we use a little bit of $h$-FIO theory.

\begin{lemma}
\label{L:exact-nf}
Let $p = \zeta^2 - z^2$ be the global quadratic form associated to the
unstable dynamical system near $(0,0)$ in our
original coordinates, and let $q = \xi x$ be the normal form for this
quadratic form.  Let 
\[
\begin{pmatrix} x \\ \xi \end{pmatrix} =
\frac{1}{\sqrt{2}} \begin{pmatrix} z + \zeta \\ \zeta -z \end{pmatrix}
\]
be the linear canonical transformation such that $\kappa^* p = -2 q$.
There is an exact unitary $h$-FIO $I : L^2(dz) \to L^2(dx)$ quantizing
$\kappa$ in the sense that the Weyl quantizations of $p$ and $q$
satisfy 
\[
I \Op_h(p)  =  \Op_h (-2q) I.
\]

\end{lemma}

\begin{remark}
We note that this Lemma asserts two things: the existence of the
$h$-FIO quantizing the canonical transformation, and a Egorov type
transformation rule (that the $h$-FIO operates by pulling back on the
level of symbols).  In addition, there is {\it no error} in the Egorov
transformation law.  The usual error in the Weyl calculus is $\O(h^2)$.

\end{remark}

\begin{proof}
Let $\hat{H} = \frac{1}{2} ((hD_x)^2 + x^2)$ be the quantum harmonic oscillator.  The
symbol of $\hat{H}$ is $H = \frac{1}{2}(\xi^2 + x^2)$ whose Hamiltonian flow
generates (clockwise) 
rotations.  That is, if we solve the Hamiltonian ODEs
\[
\begin{cases}
\dot{x} = H_\xi  = \xi \\
\dot{\xi} = -H_x = -x,
\end{cases}
\]
we get the canonical transformation
\[
\kappa_t( x, \xi ) = \begin{pmatrix} x \cos t + \xi \sin t \\ \xi \cos
  t - x \sin t \end{pmatrix}.
\]
Of course in this case $\kappa_t$ is linear, given by the (clockwise) rotation
matrix
\[
R_t = \begin{pmatrix} \cos t &  \sin t \\ -\sin t & \cos
  t \end{pmatrix}.
\]
We want to rotate the symbol $\xi^2 - x^2$ into the symbol $x \xi$ and
then check the computations on the quantum level as well.  This is a
(clockwise) 
rotation by $t = \pi /4$.  

Let $\tI(t)$ satisfy the equation
\[
\begin{cases}
hD_t \tI = - \tI \hat{H}, \\
\tI(0) = \id,
\end{cases}
\]
and $I(t)$ satisfy the equation
\[
\begin{cases}
hD_t I = \hat{H} I, \\
I(0) = \id.
\end{cases}
\]
These are adjoint equations, so $\tI(t) = I(t)^*$.  Further, the
operator $F(t) = \tI(t) I(t)$ satisfies
\[
hD_t F = -\tI \hat{H} I + \tI \hat{H} I = 0,
\]
with initial conditions $F(0) = \id$, so $F(t) \equiv \id$.  Furthermore,
the operator $G(t) = I(t) \tI(t) -\id$ satisfies the homogeneous
equation 
\[
hD_t G = [\hat{H}, G],
\]
together with the initial condition $G(0) = 0$, hence $G(0) \equiv 0$
so that $I(t)$ and $\tI(t)$ are inverses.  This shows $I(t)$ is
unitary.  We can also express these operators explicitly in terms of
harmonic oscillator projectors.  Let $\{ H_k \}$ be the normalized
eigenfunctions of the harmonic oscillator $\hat{H}$ with eigenvalues
$\lambda_k$.  Then 
\[
I (t) f = \sum_k e^{i t \lambda_k /h} \lll f, H_k \rrr H_k,
\]
and 
\[
\tI(t) g = \sum_k e^{-it \lambda_k /h } \lll H_k , g \rrr
\overline{H_k}.
\]

We now want to understand a version of the Egorov theorem for this
operator, especially at the angle of $t = \pi /4$.  Let $p = \xi^2 -
x^2$ be our initial symbol and let $b_t = \kappa_t^* p$, where
$\kappa_t$ is the rotation transformation expressed in terms of $\hat{H}$
above.  That is,
\begin{align*}
b_t( x, \xi) & = p( x \cos t + \xi \sin t , \xi \cos t - x \sin t ) \\
& = (\xi \cos t - x \sin t )^2 - (x \cos t + \xi \sin t )^2 \\
& = (\cos^2 t-\sin^2 t) (\xi^2 - x^2)  -4 x \xi \sin t \cos t.
\end{align*}
The Weyl quantization of $b_t$ is easy to compute:
\[
\Op_h (b_t) = (\cos^2 t-\sin^2 t) ((hD_x)^2 - x^2)  -4 \sin t \cos t
(x hD_x + h/2i).
\]
Differentiating with respect to $t$, we have
\[
\p_t \Op_h (b_t) = -4\cos t\sin t ((hD_x)^2 - x^2)  -4 ( \cos^ t -
\sin^2 t)
(x hD_x + h/2i).
\]
On the other hand, if we let $B(t) = I(t) ( (hD_x)^2 - x^2)\tI(t)$,
we have
\[
\p_t B(t) = \frac{i}{h} [\hat{H}, B(t) ].
\]
We want to compare this to $\p_t \Op_h(b_t)$.  We compute (after a
tedious computation)
\begin{align*}
[\hat{H}, \Op_h(b_t)] & = \hat{H} \left( (\cos^2 t-\sin^2 t) ((hD_x)^2 - x^2)  -4 \sin t \cos t
(x hD_x + h/2i) \right) \\
& \quad - \left( (\cos^2 t-\sin^2 t) ((hD_x)^2 - x^2)  -4 \sin t \cos t
(x hD_x + h/2i) \right) \hat{H} \\
& = - 4 \sin t \cos t \left(  \frac{h}{i} ( (xhD_x)^2 - x^2 ) \right)
+ (\cos^2 t - \sin^2 t ) \left( 2h^2 - 4 \frac{h}{i} x hD_x \right) \\
& = \frac{h}{i} \left( -4 \sin t \cos t ( (xhD_x)^2 - x^2 ) - 4 (
  \cos^2 t - \sin^2 t) (x hD_x + h/2i ) \right) \\
& = \frac{h}{i} \p_t \Op_h(b_t).
\end{align*}

That means the operators $B(t)$ and $\Op_h(b_t)$ satisfy the same
differential equation and agree at $t = 0$, so $B(t) \equiv \Op_h
(b_t)$.  We are interested in $t = \pi /4$, which gives 
\[
I(\pi /4) ( (hD_x)^2 - x^2) \tI(\pi /4) = -2 ( x hD_x + h/2i ).
\]

\end{proof}

We apply this Lemma locally near $(0,0)$ where our semiclassical
operator has full symbol
\[
P_{\mu,a}(z, \zeta, h) = p_{\mu,a}^0 + h p_{\mu , a}^1.
\]
Near $(0,0)$, we have $a(z) \equiv 0$, so in this neighbourhood
(recalling the form of $R(z)$ and using the notation $\mu^2 = 1 + E +
iF$ with $F = \O(h)$)
\[
p_{\mu,a}^0 = \zeta^2 - z^2 - E,
\]
and
\[
p_{\mu,a}^1 = -i \frac{F}{h} .
\]

Given the $h$-FIO constructed in Lemma \ref{L:exact-nf}, we can of
course rotate the other direction to replace the annoying $-2$ with a
$2$.  We can then smoothly
rotate back to identity outside a neighbourhood of $(0,0)$, which produces
a new $h$-FIO (still denoted by $I$) which can be extended globally on
$L^2(M)$.    
Choose a microlocally elliptic operator $e$ such that $e \equiv \id$
near $(0,0)$ on the set where we have not modified $I$.  Then 
\begin{equation}
\begin{cases}
I PI^{-1} = Q,\\
e \circ Q = 2 \left(x hD_x + \frac{h}{2i} -E/2 -iF/2 \right) \circ e.
\end{cases}
\end{equation}
We observe that, since conjugation by $I$ acts by {\it pullback} in
phase space, a rotation of $\pi /4$ counterclockwise rotates the local
dynamical system $\pi /4$ {\it clockwise} (see Figures \ref{F:phase1}
and \ref{F:phase4}).





We now are going to use this construction together with a monodromy argument to
construct quasimodes for the stationary damped wave operator.  The
point is that, since a wave packet must travel through the damping
region for some time, the incoming and outgoing coefficients are
related by a non-unitary factor.  This implies that the
quasi-eigenvalues have non-zero imaginary part.  In what follows we
will endeavour to use $(z, \zeta)$ for the original coordinates and
$(x, \xi)$ for canonical coordinates.  We will use sub- and
super-scripts of in/out to denote solutions microsupported on
stable/unstable manifolds, and $\pm$ to denote $\pm \zeta\geq 0$ in
the original coordinates.  In canonical coordinates, which we recall
begins with a linear rotation by $\pi/4$ clockwise, the $\pm$ refers
to 
$\pm x \geq 0$ (see Figure \ref{F:phase4}).  
\begin{figure}
\input{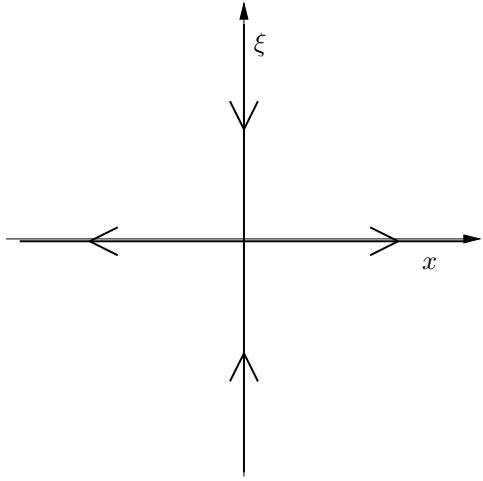}
\caption{\label{F:phase4} The local phase portrait near the hyperbolic
  fixed point $(0,0)$ in canonical coordinates $(x, \xi)$.  Observe
  the local stable/unstable manifolds have been rotated clockwise by $\pi/4$.}
\end{figure}

In our original coordinates, write 
\[
\psi^{ {in/out}}_{\pm} (z) = e^{i \phi^{ {in/out}}_{\pm}(z)
  /h} \sigma^{ {in/out}}_{\pm} (z,h),
\]
for a real phase $\phi^{ {in/out}}_{\pm}(z)$ independent of $h$
and an amplitude $\sigma^{ {in/out}}_{\pm}(z,h)$.  
Near $z = 0$, the function $a(z) \equiv 0$, so the functions
  $\psi^{ {in/out}}_{\pm}$ solve an {\it un}-damped equation
  there.  Hence we can relate these solutions near $z = 0$ to the
  model problem in canonical coordinates by conjugation using Lemma
  \ref{L:exact-nf}.

Since everything has been assumed to be symmetric about $z = 0$,
  eigenfunctions must be odd or even.  To fix one, let us assume the
  eigenfunction in which we are interested is even.  Hence
\begin{equation}
\label{E:psi-even}
\begin{cases}
\psi^{ {out}}_{+}(z) = \psi^{ {out}}_{-} (-z),  \text{ for } z
>0, \text{ and }\\
\psi^{ {in}}_{+}(-z) = \psi^{ {in}}_{-} (z), \text{ for } z
>0.
\end{cases}
\end{equation}


On the other hand, in canonical coordinates, we can solve the model
problem explicitly.  In what follows, we denote by
\[
\rho(h) = h/2i -E/2 -iF/2.
\]

 Let 
 $$ v_+^{ {out}}(x, h) = 1_{x>0} x^{- \frac i h \rho(h)}, \qquad v_-^{
   {out}} (x,h) = v_+^{ {out}} (-x, h), $$
and
\[
\hat{v}_+^{ {in}}(\xi, h) = 1_{\xi>0} x^{ -\frac i h \rho(h)}, \qquad \hat{v}_-^{
   {in}} (\xi,h) = \hat{v}_+^{ {in}} (-\xi, h), 
\]
 be the microlocal basis of the space of solutions to 
 $$ (x \partial_x + \frac i h \rho (h)) u =0. $$
These solutions are also valid in a neighbourhood of $(0,0)$,
  hence again the damping function $a$ has no effect.  Then these
  solutions can be related via the transfer matrix.

Solutions in canonical coordinates must be related to solutions
  in original coordinates via the FIO in Lemma \ref{L:exact-nf}.  The FIO
  $I$ is independent of $\rho(h)$.  Working microlocally near $(0,0)$
  and applying $I^{-1}$ and
  using that the microsupport of each of the $v^{ {in/out}}_\pm$
  is rotated counterclockwise by $\pi/4$, the resulting functions must
  be expressed as scalar multiples of the corresponding microlocal solutions
  in the original coordinates.  That is, we write 
\[
I^{-1}  v^{ {in/out}}_\pm = \gamma^{ {in/out}}_\pm
e^{i \rho^{ {in/out}}_\pm} \psi^{ {in/out}}_\pm.
\]
Here the parameters $\gamma^{ {in/out}}_\pm,
\rho^{ {in/out}}_\pm$ are real-valued, depending on $\rho(h)$.  Our
first task is to determine the $\gamma^{in/out}_+$ and
$\rho^{in/out}_+$ in terms of the spectral parameters $E$ and $F$.  

The canonical transformation in Lemma \ref{L:exact-nf} preserves the even symmetry of
all functions.  
Then
the coefficients associated to $v_\pm^{ {out}}$ (similarly ``in'')
must be the same.  Hence 
\[
 \gamma_+^{ {in/out}} e^{i \rho_+^{ {in/out}}} =
 \gamma_-^{ {in/out}} e^{i \rho_-^{ {in/out}}} .
\]
We use a trick from \cite{CdVP-II} to compute the singularities in the
phases in terms of $E$, and then find the singularities in the
amplitudes in terms of $F$.

We write our eigenfunction in original coordinates as a linear
combination (recalling the symmetry assumption \eqref{E:psi-even})
\[
\psi = \lambda_+^{out} \psi_+^{out} + \lambda_+^{in} \psi_+^{in}.
\]
We then transform $\psi$ into canonical coordinates:
\begin{align*}
I \psi & = \lambda_+^{out} I \psi_+^{out} + \lambda_+^{in}
I\psi_+^{in} \\
& = \lambda_+^{out}  \left( \gamma_+^{ {out}} e^{i \rho_+^{ {out}}}
\right)^{-1} v_+^{out} +       \lambda_+^{in} \left( \gamma_+^{ {in}} e^{i \rho_+^{ {in}}}
\right)^{-1} v_+^{in}.
\end{align*} 
As $I\psi$ is a microlocal eigenfunction near $(0,0)$, we know the
coefficients must be related by the transfer matrix.  We will use
this, together with geometric optics near $(0,0)$ to compute the
singularities in the phase and the amplitude as $E \to 0$.  We observe
that, even though the transfer matrix is a {\it matrix}, our symmetry
assumptions allow us to operate only on the $+$ components, in which
case the transfer matrix is a scalar up to $\O(h^\infty)$.  In an
abuse of notation, we will use $T$ to denote the transfer matrix and
scalar both when no confusion may arise.  

Rearranging, we have a new microlocal solution in canonical
coordinates 
\begin{align*}
(\gamma_+^{out}   \gamma_+^{ {in}}  e^{i \rho_+^{ {out}}} e^{i \rho_+^{ {in}}} )I \psi & = \lambda_+^{out}   \gamma_+^{ {in}} e^{i \rho_+^{ {in}}}
 v_+^{out} +       \lambda_+^{in} \gamma_+^{ {out}} e^{i \rho_+^{ {out}}}
 v_+^{in} \\
& = \tgamma^{out} e^{i \trho^{out}} v_+^{out} + \tgamma^{in} e^{i
  \trho^{in}} v_+^{in},
\end{align*}
where
\[
\tgamma^{out/in} = | \lambda_+^{out/in} | \gamma_+^{in/out} ,
\]
and
\[
\trho^{out/in} = \Arg ( \lambda_+^{out/in} ) + \rho_+^{in/out} .
\]
Then the transfer matrix relates the coefficients:
\[
\tgamma^{out} e^{i \trho^{out}}  = T( \rho(h)) \tgamma^{in} e^{i
  \trho^{in}} .
\]


In order to compute the changes in phase and amplitude, we compute the
geometric optics near $(0,0)$.  
We write down the WKB ansatz assuming $F = \O(h)$: 
\begin{align*}
( (hD_z)^2&  -z^2 -E -iF) e^{i \phi /h} \sigma \\
& = e^{i \phi / h } \left( \phi_z^2 -z^2 - E + \frac{h}{i} ( 2\sigma_z
  \phi_z + \phi_{zz} \sigma + \frac{ F}{h} \sigma ) - h^2 \sigma_{zz}
\right) \\
& = 0.
\end{align*}

That is, for $E>0$, the phases satisfy
the usual eikonal equations at energy $E$
\begin{equation*}
   \partial_z \phi^{in/out}_{\pm} = \pm \sqrt{E + z^2 } .
 \end{equation*}
Considering as usual only the $+$ components, then transitioning from
$z = -\epsilon$ to $z = \epsilon$, and fixing a gauge where the phases
$\phi_+^{in/out}$ agree at the gluing points $z = \mp \epsilon$, we have 
\begin{equation}
\label{E:rho-rel}
\Arg( \lambda_+^{out} ) = \Arg ( \lambda_+^{in} ) +
h^{-1} \int_{-\epsilon}^{\epsilon} \sqrt{E + z^2 } dz + \O_E (h).
\end{equation}
We can compute this area integral explicitly:
\begin{align*}
A(E) & : = \int_{-\epsilon}^{\epsilon} \sqrt{E + z^2 } dz \\
& = E \left[ \frac{ \epsilon ( \epsilon^2 + E)^{1/2}}{E} + \log \left(
    \frac{ \epsilon + (\epsilon^2 +E)^{1/2} }{\sqrt{E} } \right)
\right].
\end{align*}
We observe as $E \to 0$, the logarithmic term has a singularity of the
form
\[
-\frac{1}{2} E \log (E),
\]
which is not a $\Ci$ function.  

Since we are no longer in the self-adjoint setting (as opposed to \cite{CdVP-II}, we need also
compute how $| \lambda^{in/out} |$ changes as a function of $F$.  We
can solve the first transport equation (the terms with $h/i$):
\[
\sigma(z) = (\epsilon^2 + E)^{1/4} ( \phi'(z) )^{-1/2} \exp \left( -\frac{F}{2h}
  \int_{-\epsilon} ^z
  (\phi'(s))^{-1} ds \right).
\]
We have normalized so that $\sigma( -\epsilon ) = 1$.  
Then as $z$ goes from $- \epsilon$ to $\epsilon$, we have 
\begin{equation}
\label{E:lambda-rel}
| \lambda_+^{out} | = (\sigma ( \epsilon ) + \O_{E,F} (h) ) | \lambda^{in}_+|.
\end{equation}
Given the explicit form of $\phi'$, we can compute the integral in
$\sigma(\epsilon)$ exactly (noticing that the constants cancel at $z =
\pm \epsilon$):
\begin{equation}
\label{E:sigma-eps}
\sigma(\epsilon ) =   \exp
\left( -\frac{F}{2h} \log \left( \frac{ \epsilon + ( \epsilon^2
      +E)^{1/2} }{\sqrt{E} } \right) \right).
\end{equation}

Returning now to the transfer matrix formalism, we have
\[
\tgamma^{out} =| T( \rho(h))| \tgamma^{in}
\]
and
\[
\trho^{out}  = \Arg( T( \rho(h)) ) + 
  \trho^{in}.  
\]
Plugging in the definitions of of $\tgamma^{in/out}$ and
$\trho^{in/out}$ we have
\begin{equation}
\label{E:lambda-rel-1}
| \lambda_+^{out} | \gamma_+^{in}  = | T( \rho(h)) | | \lambda_+^{in}
| \gamma_+^{out}  
\end{equation}
and
\begin{equation}
\label{E:rho-rel-1}
\Arg ( \lambda_+^{out} ) +  \rho_+^{in}  = \Arg ( T( \rho(h))) +
\Arg ( \lambda_+^{in} ) +  \rho_+^{out}  .
\end{equation}
Using \eqref{E:lambda-rel} in \eqref{E:lambda-rel-1} and
\eqref{E:rho-rel} in \eqref{E:rho-rel-1}, we get
\begin{equation}
\label{E:gamma-rel-3}
| T( \rho(h)) |  \frac{ \gamma_+^{out} }{ \gamma_+^{in} }   =\sigma (
\epsilon ) + \O_{E,F} (h) 
\end{equation}
and
\begin{equation}
\label{E:rho-rel-3}
\Arg(T( \rho(h))) + \rho_+^{out} - \rho_+^{in} = \frac{A(E)}{h} +
\O_E(h).
\end{equation}
We have four asymptotic developments to consider.  Let
\[
\frac{ \gamma_+^{out} }{ \gamma_+^{in} }  = \sum_{p,q,r}
\gamma_{p,q,r} h^p E^q F^r,
\]
\[
\rho_+^{out} - \rho_+^{in} = \sum_{p,q,r}
\rho_{p,q,r} h^p E^q F^r,
\]
\[
E = \sum_{k}
E_{k} h^k ,
\]
and
\[
F = \sum_{k}
F_{k} h^k .
\]
All of the above sums start at $0$ except the sum for $\rho_+^{out} -
\rho_+^{in}$ must be allowed to start at $p = -1$, and $F_0 = 0$ so
that $F = \O(h)$.  

The last missing piece is to compute the asymptotics of the transfer
matrix.  From \cite{CdVP-II}, we have
\[
T( \rho(h) ) = \Phi \left( \frac{ E + iF }{2h} \right) (1 + \O( h^\infty) ),
\]
where
\[
\Phi( t ) = \frac{1}{\sqrt{2 \pi }} \Gamma ( 1/2 -it ) e^{t \pi /2}
e^{-i t \ln (h) } e^{i \pi /4}.
\]
For fixed $E>0$, the number $1/2 -iE/2h + F/2h$ has modulus going to
$\infty$ and real part positive if $F = \O(h)$ is sufficiently small
(we will see eventually that $F = o(h)$, so this poses no problem).
Hence we may apply Stirling's formula to the $\Gamma$ function:
\[
\Gamma(z) = \sqrt{\frac{2 \pi}{z} }\left( \frac{ z }{e} \right)^z ( 1 +
  \O(z^{-1})).
\]
For $E>0$, we write
\[
z = 1/2 -iE/2h + F/2h = -i\frac{E}{2h} \left(1 + i \frac{2h}{E} ( 1/2 +
  F/2h)\right),
\]
so that
\begin{align*}
\log (z) & = \log \left( -i\frac{E}{2h} \right) + \log \left(1 + i \frac{2h}{E} ( 1/2 +
  F/2h)\right) \\
& = \log ( E/2h) - i \pi /2 +  i \frac{2h}{E} ( 1/2 +
  F/2h) + 2 \frac{h^2}{E^2} (1/2 + F/2h)^2        + \O_E(h^3).
\end{align*}
Then
\begin{align*}
\Phi& \left( \frac{E + iF}{2h} \right) \\
& = \exp \Bigg( -\frac{1}{2}
  \log(z) + z \log (z) -z  +  \left( \frac{E + iF}{2h} \right)\pi /2
\\
& \quad -i \left( \frac{E + iF}{2h} \right) \ln (h)  + i \pi /4 \Bigg) ( 1 +
\O(h/E)) \\
& = \exp \Bigg( -\frac{1}{2}\Big( 
  \log ( E/2h) - i \pi /2 +  i \frac{2h}{E} ( 1/2 +
  F/2h) \\
& \quad \quad + 2 \frac{h^2}{E^2} (1/2 + F/2h)^2  + \O_E(h^3) \Big) \\
& \quad  + ( 1/2 -iE/2h + F/2h) \\
& \quad \quad \cdot \Big( \log ( E/2h) - i \pi /2 +  i \frac{2h}{E} ( 1/2 +
  F/2h) \\
& \quad \quad \quad + 2 \frac{h^2}{E^2} (1/2 + F/2h)^2  + \O_E(h^3) \Big) \\
 & \quad -( 1/2 -iE/2h + F/2h)   +  \left( \frac{E + iF}{2h}
 \right)\pi /2 \\
& \quad  \quad -i \left( \frac{E + iF}{2h} \right) \ln (h)  + i \pi /4 \Bigg)\\
& \quad \quad \quad \quad \cdot ( 1 +
\O(h/E)) .
\end{align*}
The imaginary part of the exponent is
\begin{align*}
\Arg( \Phi ) & =  -\frac{E}{2h} \log (E/2h) - \frac{h}{E} (1/2 + F/2h)^2 + \frac{F}{E} (1/2
+ F/2h) \\
& \quad +E/2h  -\frac{E}{2h} \ln (h) + \pi /4   + \O_E(h^3) \\
& = \frac{E}{2h} \left( - \log (E/2h) - \ln (h) +1 \right) + \pi /4 -
\frac{h}{E} (1/2 + F/2h)^2 \\
& \quad + \frac{F}{E} (1/2
+ F/2h)+ \O_{E,F}(h^3) \\
& = \frac{E}{2h} \left(1 - \log(E/2) \right) + \pi /4 - \frac{h}{E}
(1/2 + F/2h)^2 \\
& \quad + \frac{F}{E} (1/2
+ F/2h)+ \O_{E,F}(h^3) .
\end{align*}
The real part of the exponent is
\begin{align*}
\tau & : = \frac{F}{2h} \log(E/2h)   +\frac{F}{2h} \ln (h)    +
\frac{F h}{E^2} ( 1/2 + F/2h)^2 + \O(h^2/E^2) \\
& = \frac{F}{2h} \log (E/2) + \O(h^2 / E^2),
\end{align*}
since we have assumed $F = \O(h)$.

Reading off the
first terms in the expansion \eqref{E:rho-rel-3}, we have for
$h^{-1}$:
\begin{align*}
\frac{E}{2h} \left(1 - \log(E/2) \right)  + \frac{\rho_{-1,0,0}}{h} & =
\frac{A(E)}{h} \\
& = \frac{E}{h} \left[ \frac{ \epsilon ( \epsilon^2 + E)^{1/2}}{E} + \log \left(
    \frac{ \epsilon + (\epsilon^2 +E)^{1/2} }{\sqrt{E} } \right)
\right].
\end{align*}
The logarithmic singularity is the same on each side of this equation,
so $\rho_{-1,0,0}$ is a smooth function of $E\geq 0$:
\[
\rho_{-1,0,0} = -\frac{E}{2} ( 1 + \log(2)) + \epsilon ( \epsilon^2
+E)^{1/2} + E \log ( \epsilon + (\epsilon^2+E)^{1/2} ).
\]
The terms with $h^0$ read
\[
\rho_{0,0,0} = - \pi /4, \,\,\, E \rho_{0,1,0} = 0,
\]
and the next terms read
\[
h \rho_{1,0,0} + F \rho_{0,0,1} - \frac{h}{E} (1/2 + F/2h)^2 + \frac{F}{E} (1/2
+ F/2h) = \O_E(h).
\]
Setting
\[
\rho_{0,0,1} = \frac{1}{E} (1/2
+ F/2h)
\]
one can solve for $\rho_{1,0,0}$ to remove the remaining terms.  This
solves for the phase difference up to $\O_{E,F}(h^3)$.  The remaining
terms in the series are similarly obtained.

Reading off the first terms in the expansion \eqref{E:gamma-rel-3} for
the amplitude, 
\begin{align*}
| T( \rho(h)) | (\gamma_{0,0,0} + h \gamma_{1,0,0} + E \gamma_{0,1,0}
+ F \gamma_{0,0,1} ) & = \sigma (
\epsilon ) + \O_{E,F}(h),
\end{align*}
or
\begin{align*}
\exp & \left( \frac{F}{2h} \log (E/2)  + \O_F(h^2/E^2) \right) (\gamma_{0,0,0} + h \gamma_{1,0,0} + E \gamma_{0,1,0}
+ F \gamma_{0,0,1} )  \\
& = 
 \exp
\left( -\frac{F}{2h} \log \left( \frac{ \epsilon + ( \epsilon^2
      +E)^{1/2} }{\sqrt{E} } \right) \right) + \O_{E,F}(h).
\end{align*}
Rearranging and pulling the $h^0$ terms, we have
\[
\gamma_{0,0,0} + E \gamma_{0,1,0} = \exp \left( -\frac{F}{4h} \log (E) + \frac{F}{2h} (
  \log(2) - \log ( \epsilon + (\epsilon^2 +E)^{1/2} ) )\right).
\]
We can take $\gamma_{0,1,0} = 0$.  
Notice in this case, $\gamma_{0,0,0}$ is a smooth function of $E \to
0$ if 
\[
\frac{F}{h} \log (E)
\]
is a smooth function.  In particular, we must have $F = \O ( h | \log
(E ) | )$.  
Writing out the terms for $h^1$ we have:
\[
h \gamma_{1,0,0} + F \gamma_{0,0,1} = \O_{E,F}(h),
\]
which can be solved for any error $\O_{E,F}(h)$ (this is the error in
computing the geometric optics amplitude from $z = - \epsilon$ to $z =
\epsilon$.  This computes the change in amplitude up to $\O_{E,F}( h^2)$.
Again, the remaining terms in the series are similarly computed.

Let us return to the geometric optics construction of
$\psi^{ {in/out}}_\pm(z)$, and now compute the monodromy as $z$ goes
from $\epsilon$ to $2 \pi - \epsilon$, using the homoclinicity.  The phases satisfy the usual eikonal
equations, and we have normalized by taking all phase functions to be
$0$ at the ``gluing'' points $z = \pm \epsilon$: 
 \begin{equation}
 \begin{aligned}
   &\partial_z \phi^{in/out}_{\pm} = \pm \sqrt{ 1 + E -\frac 1 {R(z)^2} },\\
     &\phi_{+} ^{in/out} (\mp \epsilon) =0, \quad \phi_{-}^{in/out} (
     \pm \epsilon ) = 0
     .
 \end{aligned}
 \end{equation}
This choice of normalization is chosen to be compatible with the
transfer matrix computations above; the change in phase from
$-\epsilon$ to $\epsilon$ is in the coefficients $\lambda_+^{in/out}$
rather than the phases $\phi_+^{in/out}$.  
We recall that
\[
\mu^2 = 1 + E + iF,
\]
with $F = \O(h)$.  If $E>0$ or $E$ is sufficiently small, then we can
expand
\[
\mu = \sqrt{1 + E} + i\frac{  F}{2 \sqrt{1 + E}} + \O(h^2).
\]
The associated symbols $\sigma_{\pm}^{out, in}\sim \sum_k h^{k}
\sigma_{\pm, k}^{out, in}$ satisfy the transport equations
\begin{align*}
 2\partial_z \phi \partial_z \sigma_0 + \left(\phi_{zz}  -a\mu +
   \frac{F}{h} \right)
\sigma_0 & = 0; \\
 2\partial_z \phi \partial_z \sigma_k + \left(\phi_{zz}  -a\mu +
   \frac{F}{h} \right)
\sigma_k & = i \p_z^2 \sigma_{k-1}.
\end{align*}

Here we have dropped the $\pm$ and in/out notation to (slightly)
simplify the presentation.  
  This allows us to describe in the region $\pm z \geq \epsilon >0$, the geometric optics solutions 
\[
\psi^{in/out}_\pm(z).
\]
  As a consequence, we obtain the following Proposition.  
  \begin{proposition}
\label{P:sigma}
For any $\delta>0$, there exists a normalized, microlocally defined function $v(z)$ on $\reals$
satisfying the following properties:
\begin{enumerate}

\item
The function $v(z)$ is almost periodic:
\[
v(z-2 \pi ) = v(z) + \O(h^{2-\delta}),
\]
for $z$ near $\epsilon$.  

\item
The derivative of $v(z)$ is almost periodic:
\[
\p_z v(z-2 \pi ) = \p_z v(z) + \O(h^{1-\delta}),
\]
for $z$ near $\epsilon$.  

\item
The function $v$ is a quasimode:
\[
P_{\mu,a}^h v = \O(h^{2-\delta})
\]
for $z$ in a neighbourhood of $[0, 2 \pi + \epsilon ]$.

\end{enumerate}

  \end{proposition}

\begin{figure}
\input{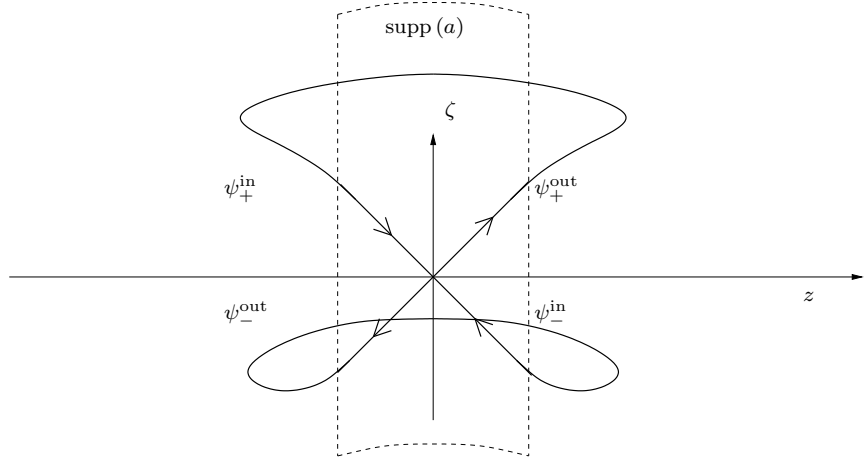}
\caption{\label{F:phase3} The global (periodic) phase portrait again,
  ``wrapped'' around $T^* \mathbb{S}^1$, 
  together with the microlocal phases of the solutions to $P_h(z,\p_z,h) u = 0$.}
\end{figure}

\begin{proof}
To construct quasi-modes on the manifold $M$, we must construct
the solutions away from $z=0$, by solving the eikonal and transport
equations above. The first equation for the symbol $\sigma_0$ has an
explicit solution:
\begin{align*}
& \sigma_{0,0,0}(z)  \\
& \quad = (\epsilon^2 +E)^{1/4} (\phi'(z))^{-1/2} \exp \left(
  -\frac{F}{2h} \int_\epsilon^z (\phi'(s))^{-1} ds + \frac{ \mu}{2}
  \int_\epsilon^z a(s) (\phi'(s))^{-1} ds \right).
\end{align*}


Since $\phi'$ is even, we have $\phi'(2 \pi - \epsilon) =
\phi'(\epsilon)$.  Hence $\sigma_{0,0,0} ( \epsilon ) = 1$, and
\[
\sigma_{0,0,0} ( 2 \pi - \epsilon ) = \exp \left( -c_0(E) \frac{F}{2h} +
  \frac{ \mu}{2} c_1(a,E)
\right),
\]
where
\[
c_0 (E) =  \int_\epsilon^{2\pi - \epsilon} (2\phi'(s))^{-1} ds,
\]
and
\[
c_1(a,E) = \int_0^{2\pi } a(s) (2\phi'(s))^{-1} ds,
\]
if $\epsilon>0$ is sufficiently small that $a(z) \equiv 0$ for $| z |
\leq \epsilon$.


We know that the solutions $\psi^{in}_\pm$ must be related to the
solutions $\psi^{out}_\pm$ by monodromy.  That is, there is an
operator $e^{\tK_\epsilon}$ such that
\[
\begin{pmatrix} \psi^{in}_+ ( - \epsilon) \\\psi^{in}_- (
  \epsilon) \end{pmatrix} = e^{\tK_\epsilon} \begin{pmatrix} \psi^{out}_+
  (\epsilon) \\\psi^{out}_- (-\epsilon) \end{pmatrix}
\]
Our assumption that these functions have an even symmetry reduces this
to the scalar equation
\[
\psi^{in}_+ ( - \epsilon) = e^{\tK_\epsilon} \psi^{out}_+
  (\epsilon) .
\]
But we can compute the evolution of $\psi^{out}_+(z)$ through the
damping using our geometric optics construction and match it with
$\psi^{in}_+$ to find $e^{\tK_\epsilon}$.  That is, we have
\[
\psi^{out}_+ ( 2 \pi - \epsilon ) = e^{i \phi^{out}_+ ( 2 \pi -
  \epsilon) /h } \sigma^{out}_+ ( 2 \pi - \epsilon ),
\]
and we can compute the phase and principal symbol explicitly.  We have
\begin{align*}
\phi^{out}_+ ( 2 \pi -
  \epsilon) & = \int_\epsilon^{2 \pi - \epsilon} \p_z \phi^{out}_+ (z) dz \\
& = \int_0^{2 \pi } \p_z \phi^{out}_+ (z) dz - \int_{ -
  \epsilon}^{\epsilon } \p_z \phi^{out}_+ (z) dz \\
& = B(E) - A(E)
\end{align*}
where 
\[
B(E) = \int_0^{2 \pi } \sqrt{ 1 + E - R^{-2}(z) } dz, 
\]
and
\begin{align*}
A(E) & = \int_{-\epsilon}^\epsilon \sqrt{1 + E - R^{-2}(z) } dz \\
& = \int_{-\epsilon}^\epsilon \sqrt{ E +z^2 } dz
\end{align*}
as before.  

Similarly,
\[
\sigma^{out}_{0,0,0} ( 2 \pi - \epsilon ) = e^{-c_0(E) \frac{F}{2h} +
  \frac{ \mu}{2} c_1(a,E)} \sigma^{out}_{0,0,0} ( \epsilon ),
\]
so that the principal part of the monodromy is computed
\begin{align}
\lambda_+^{in} & = \lambda_+^{in} e^{i \phi^{in}_+ (- \epsilon ) /h} \sigma^{in}_{0,0,0} (
- \epsilon ) \notag \\
& = \lambda_+^{out} e^{i \phi^{out}_+ (2 \pi- \epsilon ) /h} \sigma^{out}_{0,0,0} (
2 \pi - \epsilon ) \notag \\
& = \lambda_+^{out} 
e^{i (B(E) - A(E)) / h } 
 e^{-c_0(E) \frac{F}{2h} +
  \frac{ \mu}{2} c_1(a,E)} \sigma^{out}_{0,0,0} ( \epsilon ) \notag \\
& = \lambda_+^{out} 
e^{i (B(E) - A(E)) / h } 
 e^{-c_0(E) \frac{F}{2h} +
  \frac{ \mu}{2} c_1(a,E)}. \label{E:mono-101}
\end{align}

Let us expand the amplitudes $\sigma^{in/out}$ in asymptotic
developments:
\[
\sigma^{in/out}(z) = \sum_{p,q,r} \sigma^{in/out}_{p,q,r} h^p E^q
F^r.
\]
Then
\begin{align*}
\lambda_+^{in}&  e^{i \phi^{in}_+ (- \epsilon ) /h} \Big(
  \sigma^{in}_{0,0,0} (-\epsilon) + h \sigma^{in}_{1,0,0}(-\epsilon) +
  E \sigma^{in}_{0,1,0} (-\epsilon) \\
& \quad   + F \sigma^{in}_{0,0,1} (-\epsilon)+ \O ( E^2 + h^2 ) \Big) \\
& = \lambda_+^{out} e^{i \phi^{out}_+ (2 \pi- \epsilon ) /h} \Big(
  \sigma^{out}_{0,0,0} (2 \pi -\epsilon) + h \sigma^{out}_{1,0,0}(2
  \pi -\epsilon) \\
& \quad +
  E \sigma^{out}_{0,1,0} (2 \pi -\epsilon)
  + F \sigma^{out}_{0,0,1} (2 \pi -\epsilon)+ \O ( E^2 + h^2 ) \Big)  
\end{align*}
Plugging in \eqref{E:mono-101} for the principal terms, we have
\begin{align}
\lambda_+^{in} & \left(
  1 + h \sigma^{in}_{1,0,0}(-\epsilon) +
  E \sigma^{in}_{0,1,0} (-\epsilon)
  + F \sigma^{in}_{0,0,1} (-\epsilon)+ \O ( E^2 + h^2 ) \right) \notag
\\
& = \lambda_+^{out} e^{i (B(E) - A(E)) /h} \Big( e^{-c_0(E) \frac{F}{2h} +
  \frac{ \mu}{2} c_1(a,E)}
   + h \sigma^{out}_{1,0,0}(2
  \pi -\epsilon) \notag \\
& \quad +
  E \sigma^{out}_{0,1,0} (2 \pi -\epsilon)
  + F \sigma^{out}_{0,0,1} (2 \pi -\epsilon)+ \O ( E^2 + h^2 )
\Big). \label{E:mono-102}
\end{align}
We have computed already that
\[
\frac{\lambda_+^{out}}{\lambda_+^{in}} = (\sigma(\epsilon) + R_1 )
\exp( i A(E) /h + iR_2),
\]
where $\sigma(\epsilon)$ was computed in \eqref{E:sigma-eps} and $R_1,
R_2 = \O(h)$.   Set $e^{K_\epsilon} = e^{-c_0(E) \frac{F}{2h} +
  \frac{ \mu}{2} c_1(a,E)}$.  
Solving for $\lambda_+^{out}/ \lambda_+^{in}$ in
\eqref{E:mono-102}, we have 
\begin{align}
& (\sigma(\epsilon) + R_1 )
\exp( i A(E) /h + iR_2) \notag \\
& \quad = \left(
  1 + h \sigma^{in}_{1,0,0}(-\epsilon) +
  E \sigma^{in}_{0,1,0} (-\epsilon)
  + F \sigma^{in}_{0,0,1} (-\epsilon)+ \O ( E^2 + h^2 ) \right) \notag
\\
& \quad \quad \cdot e^{i (-B(E) + A(E)) /h} \Big( e^{K_\epsilon} 
   + h \sigma^{out}_{1,0,0}(2
  \pi -\epsilon) \notag \\
& \quad \quad \quad +
  E \sigma^{out}_{0,1,0} (2 \pi -\epsilon)
  + F \sigma^{out}_{0,0,1} (2 \pi -\epsilon)+ \O ( E^2 + h^2 )
\Big)^{-1} \notag \\
& = e^{i (-B(E) + A(E)) /h} e^{-K_\epsilon} \Big( 1 + h
  (\sigma^{in}_{1,0,0}(-\epsilon) - e^{-K_\epsilon}
  \sigma^{out}_{1,0,0}(2 \pi -\epsilon)   ) \notag \\
& \quad \quad + E
  (\sigma^{in}_{0,1,0}(-\epsilon) - e^{-K_\epsilon}
  \sigma^{out}_{0,1,0}(2 \pi -\epsilon)   ) \notag \\
& \quad \quad + F
  (\sigma^{in}_{0,0,1}(-\epsilon) - e^{-K_\epsilon}
  \sigma^{out}_{0,0,1}(2 \pi -\epsilon)   ) + \O(E^2 + h^2 )\Big). \label{E:mono-103}
\end{align}

Comparing phases on both sides of \eqref{E:mono-103}, we require 
\[
 \frac{ A(E) }{h} + \tR_2 =  -\frac{B(E)}{h} + \frac{A(E)}{ h} + 2 \pi
 k,
\]
for integer $k$, or
\[
B(E) + h \tR_2 = 2 \pi k h.
\]
Here the error $\tR_2 = R_2 +i Fc_1(a,E) /4\sqrt{1 +E} + \O(h^2)$ consists of all
of the terms in the amplitude of order $h$ or smaller.  
We observe that this Bohr-Sommerfeld type quantization condition is
independent of the gluing point $\epsilon$, and gives
a discrete choice of values of $E$.  In particular, this equation can
be solved for $E>0$, $E \sim h$, as an asymptotic series as described
previously.  For such a value of $E$, we compare the amplitudes on
each side of \eqref{E:mono-103}:
\begin{align}
\sigma & (\epsilon) + R_1 \\
& = e^{-K_\epsilon} \Big( 1 + h
  (\sigma^{in}_{1,0,0}(-\epsilon) - e^{-K_\epsilon}
  \sigma^{out}_{1,0,0}(2 \pi -\epsilon)   ) \notag \\
& \quad \quad + E
  (\sigma^{in}_{0,1,0}(-\epsilon) - e^{-K_\epsilon}
  \sigma^{out}_{0,1,0}(2 \pi -\epsilon)   ) \notag \\
& \quad \quad + F
  (\sigma^{in}_{0,0,1}(-\epsilon) - e^{-K_\epsilon}
  \sigma^{out}_{0,0,1}(2 \pi -\epsilon)   ) + \O(E^2 + h^2
  )\Big). \label{E:mono-amp}
\end{align}

Recalling \eqref{E:sigma-eps} and the definition of $e^{K_\epsilon}$,
we have the leading order equation
\[
-\frac{F}{2h} \log \left( \frac{ \epsilon + ( \epsilon^2
      +E)^{1/2} }{\sqrt{E} } \right) = -c_0(E) \frac{F}{2h} +
  \frac{ \sqrt{1 +E}}{2} c_1(a,E).
\]
Now if $E = \O(h)$, then we have already found that $F = \O( h / |
\log (E) | ) = \O(h/ | \log (h) | )$, so the term with $c_0$ is
$o(1)$.  Hence we want to solve
\[
\frac{c_1(a,E) \sqrt{1 +E} }{2} = - \frac{F}{4h} \log (E),
\]
or
\[
F = \frac{2 h c_1(a,E) }{ | \log (E) |} (1 + o(1)).
\]
This determines $F$.  Expanding $R_1$ in an asymptotic series in
$h,E,F$, we can solve for the initial conditions on the lower order
amplitude terms in \eqref{E:mono-amp}.  
This completes the proof of Proposition \ref{P:sigma}.

\end{proof}

We now show that Theorem \ref{theo-quasi} follows from Proposition
\ref{P:sigma}.  
\begin{proof}[Proof of Theorem \ref{theo-quasi}]
Let $v$ be as in the statment of Proposition \ref{P:sigma}.  The main
problem is that $v$, as constructed, does not live on the circle but on
the real line.  Nevertheless, since $v$ is almost periodic, we will
glue $v$ together with a shift by $2\pi$ to construct an honestly
periodic function.  
Choose $\chi \in \Ci_c([0, 2 \pi + \epsilon])$, $0 \leq \chi \leq 1$,
with $\chi(z) \equiv 1$ for $z \in [\epsilon, 2 \pi]$, and satisfying
\[
\chi(z) + \chi(z + 2 \pi) = 1 \text{ for } z \in [0, \epsilon].
\]
The function 
\[
u(z) = \sum_k \chi(z+ 2 \pi k) v(z + 2 \pi k )
\]
is $2 \pi$-periodic, so it is determined on any interval of length $2
\pi$, say $z \in [\epsilon, 2 \pi + \epsilon]$.  

On the interval $[\epsilon, 2 \pi]$, $\chi(z) \equiv 1$, and for any
$k \neq 0$, we have $\chi(z + 2 \pi k) = 0$.  Hence for $z \in [
\epsilon, 2 \pi]$, $u(z) = v(z)$.  On the other hand, for $z \in [2
\pi, 2 \pi + \epsilon]$, 
\begin{align}
u(z) & = \chi(z) v(z) + \chi(z - 2 \pi) v(z - 2 \pi) \notag  \\
& = \chi(z) v(z) + (1 - \chi(z )) v(z - 2 \pi) , \label{E:chi-prop}
\end{align}
by construction of $\chi$.  

We compute: 
\[
P_{\mu,a}^h u = \chi (z) P_{\mu,a}^h v(z) +  \chi(z-2 \pi ) P_{\mu,a}^h
v(z - 2 \pi ) + [P_{\mu,a}^h, \chi] (v(z) - v(z - 2 \pi ) ),
\]
where we have used \eqref{E:chi-prop} in the commutator term.  
The commutator has terms with $h^2 \chi''$ and $h \chi' h \p_z$.
Since $\chi'$ and $\chi''$ are both supported near $z = \epsilon$, we use the continuity conditions in Proposition
\ref{P:sigma} to get
\[
[P_{\mu,a}^h, \chi] (v(z) - v(z - 2 \pi ) ) = \O(h^2 h^{2 - \delta} )
+ \O(h^2 h^{1 - \delta} ).
\]
That $v(z)$ and $v(z - 2 \pi)$ are both quasimodes as in Proposition \ref{P:sigma} then implies
\[
P_{\mu,a}^h u = \O( h^{2 - \delta}).
\]
This is Theorem \ref{theo-quasi}.

\end{proof}

\subsection{Quasimodes imply sub-exponential damping}
In this section we prove Theorem~\ref{th.1}. Let us consider a sequence of quasimodes $\{u_j \}$
and
quasi-eigenvalues $\{ \tau_j \}$ (as constructed above) satisfying
\[
(-\tau_j^2 - \Delta_g + i \tau_j a(x) ) v_j = R_j,
\]
with 
\begin{equation}\label{eq.est}
\| R_j \|_{L^2} = \mathcal{O}_\epsilon (| \tau_j|^{ \epsilon}) \| v_j \|_{L^2}, \qquad \| v_j\|_{L^2} =1, \qquad \| \nabla_x v_j\|_{L^2} \sim |\tau_j| 
\end{equation}
for any $\epsilon>0$, and such that 
\begin{equation}\label{eq.lim}
\Re \tau_j \rightarrow + \infty,\qquad \Im \tau_j \sim c\log^{-1} ( \Re \tau_j), \qquad  j \rightarrow + \infty , \qquad c>0
\end{equation}
Let us consider $u_j$ the solution to the damped wave equation~\eqref{wave-equation-0} with initial data 
$$ (u_0=v_j, u_1= i\tau_j v_j).$$
 \begin{lemma}
There exists $C>0$ such that  for any $T>0$, 
$$\| u_j - e^{it \tau_j} v_j\|_{L^{\infty} ((0,T); H^1( M))}+ \| \partial_t u_j - i \tau_j e^{it \tau_j} v_j\|_{L^{\infty} ((0,T); L^2( M))}\leq C \log( |\tau_j|) \| R_j\|_{L^2}
$$
\end{lemma}  
\begin{remark}
In the following proof, we consider non-real quasimodes.  Of course
one can prove the same result for real-valued functions by taking the
real or imaginary parts of the quasimodes constructed below.  
\end{remark}

\begin{proof}
Indeed, the function $w_j = u_j- e^{it\tau_j} v_j$ satisfies 
$$ (\partial_t^2 - \Delta +  a(x) \partial_t ) w_j = e^{it \tau_j }R_j 
$$
and from the Duhamel formula and~\eqref{eq.lim} we get (here we use that the semi-group associated to the damped wave equation is a semi-group of contractions)
\begin{multline}
\| u_j - e^{it \tau_j} v_j\|_{L^{\infty} ((0,T); H^1( M))}+ \| \partial_t u_j - i \tau_j e^{it \tau_j} v_j\|_{L^{\infty} ((0,T); L^2( M))}\\
\leq\int_0 ^t \| e^{is\tau_j} R_j \|_{L^2} ds\leq \int_0 ^t | e^{- \frac{ s} {\log( |\tau|_j)}}|  ds \| R_j\|_{L^2},
\end{multline} 
which proves the lemma. 
\end{proof}

We are now ready to prove Theorem \ref{th.1}.  

\begin{proof}[Proof of Theorem \ref{th.1}]
Let $\delta>0$ be the derivative loss in the statement of Theorem
\ref{th.1}.   That is, for our choice of initial data $(u_0=v_j, u_1=
i\tau_j v_j)$, we have
\[
\| (u_0, u_1) \|_{H^{1 + \delta} \times H^\delta}^2 \sim | \tau_j |^{2
  + 2 \delta}.
\]

Using the previous Lemma, 
we deduce that for any $t>0$,
$$ E(u_j) (t) = | \tau_j|^2 \Bigl( e^{-2c\frac t { \log( |\tau_j|)}} + \mathcal{O} \bigl( |\tau_j|^{2\epsilon-2} \log^2( |\tau_j|)\bigr)\Bigr)\leq |\tau_j|^{2+ 2\delta} f(t)
$$ 
For fixed $t$, we optimize the estimate by choosing $j$ so that $ t \sim \frac{\epsilon} { 4c }\log^2 (|\tau_j|)$ and  we get 
$$ f(t)\geq \frac{|\tau_j|^{-2\delta}} 2 e^{- \frac \epsilon 2 \log( |\tau_j|)},$$
or equivalently,
$$   f(t)\geq e^{-c_{\delta, \epsilon} \sqrt{t}}.$$

\end{proof}

\section{Overdamping: the case of perfect geometric control}
\label{S:overdamping-I}
In this section, we prove that the presence of stronger damping
does not hurt anything in the case of perfect geometric control.
Specifically, we study the following problem.  Let $\Omega \subset
\reals^n$ be a bounded domain with smooth boundary.  Let $u$ be a
solution to the following over-damped wave equation:
\begin{equation}
\label{wave-equation-1}
\left\{ \begin{array}{l}
\left( \partial_t^2 - \Delta  -\div a(x) \nabla \partial_t \right) u(x,t) = 0, \quad (x,t) \in X \times (0, \infty) \\
u(x,0) = 0, \quad \partial_t u(x,0) = f(x).
\end{array} \right.
\end{equation}
We assume $a$ controls $\Omega$ geometrically:
\begin{equation}
\label{RT}
\left\{ \begin{array}{l} \text{There exists a time }T>0  \text{ such
      that for every } 
  (x, \xi) \in S^*\Omega , 
 \text{the (unit speed)} \\ \text{geodesic beginning at } (x,\xi), \,\, \gamma(t),  \text{ meets } \{ a>0\}     \text{ for some } |t| \leq T.
    \end{array} \right.
\end{equation}
We also require some estimates on $a$ near the set where $a = 0$.  We
assume there exists $k>2$ such that
\[
| \p^\alpha a | \leq C_\alpha a^{(k-|\alpha|)/k}, \,\, | \alpha | \leq
2.  
\]
This follows, for example, if there exists a defining function $x$ for
$\{ a > 0 \}$ such that $\p_x^3 a \geq 0$ (see \cite[Lemma 3.1]{BuHi-dw}).

Then we have the following Theorem.

\begin{theorem}
\label{T:wave-decay-1}
Let $u$ be a solution to \eqref{wave-equation-1} and assume $a \in
\Ci(\Omega)$ satisfies \eqref{RT}.  Then there exists a constant $C>0$
such that
\begin{equation}
\label{E:energy-decay-1}
\| \partial_t u \|_{L^2(\Omega)}^2 + \| \nabla u \|_{L^2(\Omega)}^2
\leq C e^{-t/C} \| f \|_{L^2(\Omega)}^2.
\end{equation}
\end{theorem}

The proof uses semiclassical defect measures and a contradiction
argument to prove a resolvent estimate, similar to the proof of
\cite[Theorem 5.9]{Zwo-book}.  In order to prove the resolvent estimate,
we first formally cut off in time and take the Fourier transform to
get the equation
\begin{equation}
\label{wave-equation-2}
\tP(\lambda)\hat{u}(x, \lambda) :=(-\lambda^2 - \Delta - i\lambda \div  a(x)
\nabla ) \hat{u}(x,
\lambda) = f.
\end{equation}
We introduce a semiclassical parameter $h = (\Re \lambda )^{-1}$, and
set 
\[
\lambda^2 = \frac{z}{h^2},
\]
and upon rescaling are led to study the semiclassical equation
(abusing notation slightly)
\begin{equation}
\label{E:sc-dw-1}
P(z,h) u = g,
\end{equation}
where
\[
P(z,h) = (-h \div(1 + i h^{-1} \sqrt{z} a ) h \nabla - z ) ,
\]
and
\[
g = h^2 f.
\]
We recall the definition of the semiclassical Sobolev spaces on
$\Omega$ for integer $r$:
\[
\| u \|_{\Hsc^r (\Omega )}^2 = \sum_{|\alpha| \leq r} \| (h
D_x)^\alpha u \|_{L^2(\Omega )}^2.
\]
We have the following resolvent estimate.
\begin{proposition}
\label{P:sc-dw-prop-1}
There exist constants $h_0>0$, $\alpha >0$, and $C>0$ such that for $0
< h \leq h_0$ and 
\[
z \in [1 - \alpha, 1 + \alpha] + i(-\infty , h/C],
\]
the operator $P(z,h)$ is invertible as an operator $\Hsc^1(\Omega) \to
L^2(\Omega)$ and
\[
\| P(z,h)^{-1} g \|_{\Hsc^1(\Omega)} \leq \frac{C}{h} \| g
\|_{L^2(\Omega)}.
\]
\end{proposition}

\begin{proof}
We have to prove there is a range of $z$ as in the proposition such
that if $u$ satisfies \eqref{E:sc-dw-1}, then
\begin{equation}
\label{E:res-est-1}
\| u \|_{L^2} + \| h \nabla u \|_{L^2} \leq \frac{C}{h} \| g \|_{L^2}.
\end{equation}

We first record some a priori estimates which we will use later in the
proof.  We multiply \eqref{E:sc-dw-1} by $\bar{u}$, integrate by
parts, recall $h =  (\Re \lambda)^{-1}$ and $z = h^2 \lambda^2$, and take real and imaginary parts to get the following two
identities:
\begin{equation}
\label{E:apriori-1a}
  \int | h \nabla u |^2 dx 
-\Im \sqrt{z} h^{-1} \int a | h \nabla u |^2 dx 
-\Re z \int | u |^2 dx 
= \Re \int g \bar{u} dx,
\end{equation}
and
\begin{equation}
\label{E:apriori-1b}
h^{-1}\Re\sqrt{z} \int a | h \nabla u |^2 dx 
-  \Im z \int | u |^2 dx 
=
\Im \int g \bar{u} dx.
\end{equation}

Now for the purpose of deriving a contradiction, assume
\eqref{E:res-est-1} is false, and let $u_n$ be a sequence in $\Hsc^1$
satisfying 
\[
P(z_n,h_h) u_n = g_n,
\]
with $h_n \to 0$, $\Re (z -1) = o(1)$, $\Im z = o(h_n)$, 
\[
\| u_n \|_{L^2} + \| h \nabla u_n \|_{L^2} = 1, 
\]
and
\begin{equation}
\label{E:gn-assump-1}
\| g_n \|_{L^2} = o(h).
\end{equation}

The damping term $ih^{-1} (h \div   \sqrt{z} a  h \nabla) $ in
\eqref{E:sc-dw-1} is too large to control at first inspection, so,
following \cite{BuHi-dw}, we
introduce a cutoff to the set where $a \leq h$ to control this term.
Choose $\chi \in \Ci(\reals)$, $\chi \equiv 1$ near $0$ with small
support, and let 
\[
v_n = \chi(a(x)/h) u_n, \,\, w_n = u_n - v_n.
\]
That is, $v_n$ is the part of $u_n$ localized to the set where $a \leq
h/C$ and $w_n$ is the complement.  We examine two cases and prove a
contradiction in each case.  

{\bf Case 1.}  Assume there is a subsequence $\{ w_{n_k} \}$ of the $w_n$ and a real
number $\eta>0$ independent of $h$ so that $\| w_{n_k} \|_{\Hsc^1}
\geq \eta$.  Dropping the sequence notation and renormalizing in
$\Hsc^1$, we consider $w$ satisfying 
the following equation:
\begin{equation}
\label{E:sc-dw-2}
P(z,h) w = (1-\chi(a/h)) g + [P(z,h), (1-\chi(a/h))]u ,
\end{equation}
where
\[
P(z,h) = (-h \div(1 + i h^{-1} \sqrt{z} a ) h \nabla - z ) 
\]
as before.  We claim the right hand side is still $o(h)$ in $L^2$.
The first term is clearly $o(h)$ since $g$ is and we have multiplied
$w$ by a bounded constant.  For the second term,
choose coordinates so that $x$ is a defining function for the support
of $a$, so that $a = \O(x^k)$ for some $k$ sufficiently large and
$\nabla a = \O(x^{k-1})$.  From this we have that on the set where $a
= \O(h)$, $\nabla a = \O(x^{k-1}) = \O((x^k)^{(k-1)/k}) = \O(h^{(k-1)/k})$.
Then taking the commutator gives 
\begin{align}
[P(z,h), (1-\chi(a/h))]u  = & h\div [(1 + i h^{-1} \sqrt{z} a) \chi'(a/h)
(\nabla a) u] \notag \\
& + \sum_j [h \partial_j (1 + i h^{-1} \sqrt{z} a),
\chi(a/h)] h \partial_j u \notag \\
 = & i \sqrt{z} ( \chi'(a/h) | \nabla a|^2 u +  a \chi''(a/h) h^{-1}  |
\nabla a |^2 u \notag \\
& + a \chi'(a/h) (\Delta a) u + 2 \sum_j a h^{-1}
\chi'(a/h) \partial_j a h \partial_j u  ) . \label{E:P-chia-comm}
\end{align}
To estimate the first two terms, we use
\[
\| |\nabla a |^2 u \|_{L^2} \leq C h^{2(k-1)/k} \| u \|_{L^2} =
o(h)
\]
if $(k-1)/k >1/2$ since $\| u \|_{\Hsc^1}$ is bounded.  For
the third term, we use that, on the support of $\chi'(a/h)$, $a \sim
h$, so $a \Delta a = \O(h^{1 + (k-2)/k}) = o(h)$ provided $k-2>0$.
For the last term, we use again that $a \sim h$ on the support of
$\chi'(a/h)$ so that
\begin{align*}
\| |a h^{-1}  \chi'(a/h) \nabla a | | h \nabla u | \|_{L^2} & \leq C
h^{(k-1)/k} \| h \nabla u \|_{L^2( \{ a \sim h \})} \\
& \leq C h^{-1/2 + (k-1)/k} \left( \int a | h \nabla u |^2 dx
\right)^{1/2} \\
& \leq C h^{(k-1)/k } \left( \Im \int g \bar{u} dx \right)^{1/2} + o(h)\\
& = o(h^{1/2}) h^{(k-1)/k },
\end{align*}
where we have used the a priori estimates, the fact that $u$ is
bounded, and that $g = o(h)$ in $L^2$.  Since we have already
assumed $(k-1)/k >1/2$, every term in the commutator is $o(h)$ as
claimed.  

We now have functions $w$ and $\tg$ such that $w$ is normalized in
$\Hsc^1$, $\| \tg \|_{L^2} = o(h)$, and $P(z,h) w = \tg$.  Plugging
into the a priori estimate \eqref{E:apriori-1b}, and using that $w$ is
supported where $a \geq h/C$ and $\Re z \sim 1$, we get 
\begin{align*}
\int |h \nabla w |^2 dx & \leq C \Re z h^{-1} \int a | h \nabla w |^2 dx \\
& = C \Im z \int | w |^2 dx + \Im \int \tg \bar{w} dx \\
& = o(h),
\end{align*}
since $\Im z$ and $\tg$ are both $o(h)$.  Now plugging this estimate into the a
priori estimate \eqref{E:apriori-1a} we get
\begin{align*}
\int | w |^2 dx & \leq C \Re z \int | w |^2 dx \\
& = C \left( \int (1  - \Im \sqrt{z} h^{-1}a)| h \nabla w |^2 dx -
  \Re \int \tg \bar{w} dx \right) \\
& = o(h).
\end{align*}
All told then we have shown $\| w \|_{\Hsc^1}^2 = o(h)$, which is a contradiction.

{\bf Case 2.}  We now assume there is a subsequence $\{ v_{n_k} \}$ of the $v_n$ and a real
number $\eta>0$ independent of $h$ so that $\| v_{n_k} \|_{\Hsc^1}
\geq \eta$.  Dropping the sequence notation and renormalizing in
$\Hsc^1$, we consider $v$ satisfying 
the following equation:
\begin{equation}
\label{E:sc-dw-3}
P(z,h) v = \chi(a/h) g + [P(z,h), \chi(a/h)]u.
\end{equation}
We have already computed the commutator is $o(h)$, so as in Case 1 we
consider \eqref{E:sc-dw-3} with the right hand side replaced by a
function $\tg = o(h)$ in $L^2$.  We claim again there is a
contradiction.  For this we construct semiclassical defect measures
for solutions to this equation.

We consider a slightly more general operator:
\[
\tP(z,h) = (-h \div(1 + i  \sqrt{z} b) h \nabla - z ) ,
\]
where $b$ is a bounded, non-negative function of $x$.  Assume there is
an $h$-dependent family of functions $v$ satisfying $\|v \|_{\Hsc^1} = 1$, and
\begin{equation}
\label{E:tP=o}
\tP(z,h) v = o(h).
\end{equation}

Let $\mu$ be the semiclassical defect measure associated to the
sequence $u_n$.  We claim the measure $\mu$ has the following
properties:
\begin{align}
& \text{ (i) } \supp \mu \subset \{ |\xi|^2 =1 \} \cap \{ b = 0 \},
\text{ and} \notag \\
& \text{(ii) } \mu \text{ is invariant under the geodesic flow}. \label{E:mu-prop}
\end{align}
To prove \eqref{E:mu-prop}(i), we use elliptic regularity: if
\[
p = |\xi|^2(1 + ib(x)) -1
\]
is the principal symbol of $\tP(z,h)$ and
$a(x,\xi) \in \Ci_c(T^*\Omega)$ is supported away from $\{ | \xi|^2 =
1 \} \cap \{ b = 0 \}$, then we can find $\chi \in \Ci_c(T^*\Omega)$
so that $\supp \chi \cap \supp a = \emptyset$ and
\begin{align*}
| p + i\chi(x,\xi)|\xi|^2 | & = 
\left| | \xi |^2-1 + i( b(x) + \chi(x,\xi)) | \xi |^2
\right|\\
 & \geq \lll \xi \rrr^2/C.
\end{align*}
Then 
\[
\frac{ap}{p+i \chi | \xi|^2} - a = \frac{-ia \chi}{p+i\chi|\xi|^2} = 0,
\]
and the symbol calculus combined with \eqref{E:tP=o} implies the
support properties of $\mu$ (see \cite[Theorem 5.3]{Zwo-book}).

To prove \eqref{E:mu-prop}(ii), we take $A \in \Ci_c(T^*\Omega)$ and
compute the commutator:
\begin{align*}
h^{-1} \lll [-h^2\Delta -z,A] v ,v\rrr & =  h^{-1} \lll Av,
(-h^2\Delta -\bar{z})
v \rrr - h^{-1}\lll (-h^2\Delta -z) v, A^* v \rrr  \\
& =  h^{-1} \lll Av, {\tg} +i(h\sqrt{z} \div b(x) h \nabla +2 \Im z)v
\rrr \\
& \quad  - h^{-1}
\lll \tg + ih\sqrt{z} \div b(x) h \nabla v, A^* v \rrr \\
& =:  h^{-1} \lll A v , \tg + 2 \Im z v \rrr - h^{-1} \lll \tg, A^* v \rrr + A_1 + A_2 \\
& =  o(1) + A_1 + A_2.
\end{align*}
To estimate $A_1$, we integrate by parts and take yet another
commutator to get
\begin{align*}
\left| h^{-1} \int Av h\div b(x) h \nabla \bar{v} dx \right| \leq  &
\left| - h^{-1} \int  (A
b^{1/2} h
\nabla v)
(b^{1/2} h \nabla \bar{v}) dx \right| \\ & +\left| h^{-1} \int h \nabla b^{1/2} [b^{1/2}h \nabla, A] v
\bar{v} dx \right| \\
\leq & C h^{-1} \int b | h \nabla v |^2 dx + \left| \int B(x,hD_x) v
  \bar v dx \right| 
\end{align*}
for a compactly supported, zero order symbol $B(x, \xi)$ which is supported in $\{ b >0
\}$.  The first term is $o(1)$ by the a priori estimates for $\tP$,
and the second term is $o(1)$ by the support properties of $\mu$
proved in \eqref{E:mu-prop}(i).  The estimate for $A_2$ is similar.  Hence
\[
\int_{T^*\Omega} \{ | \xi |^2 -1 ,A(x,\xi) \} d \mu = 0,
\]
or $\mu$ is flow-invariant as claimed.

Now we return to the problem at hand where $b = h^{-1} a \tchi(a/h)$,
where $\tchi$ is a compactly supported smooth function such that
$\tchi \equiv 1$ on $\supp \chi$, where $\chi$ is the cutoff for the
family $v = \chi(a/h) u$.  Using the standard argument to ``average
over geodesics'' (see, for example, \cite[Theorem 5.9]{Zwo-book}), we
conclude that, under the assumption of perfect geometric control, the
sequence $v_{n_k} = o(1)$ in $\Hsc^1$, which is a contradiction.

Hence returning to the original sequence, before localizing to $\{ a
\leq h/C \}$, we have
\[
\| u_n \|_{\Hsc^1(\Omega)}  = o(1),
\]
which is a contradiction to the normalization of $u_n$.

\end{proof}

The proof of Theorem \ref{T:wave-decay-1} now follows exactly as
the proof of \cite[Theorem 5.10]{Zwo-book}.
\qed

\section{Overdamping: the case of imperfect control}
\label{S:overdamping-II}

In this section, our assumption is that $\Omega$ is a Euclidean domain
outside a compact set $\tV$ and that $a$ controls $\Omega$
geometrically outside a subset $V \subset \tV$.  We further make what
amounts to a ``black box'' assumption, that if we continue $\Omega$ to
a scattering manifold then the semiclassical resolvent with absorbing
potential satisfies a
polynomial bound in an $h$ sized strip.  Then using the black box
framework of \cite{BuZw-bb} we have an estimate for a damped wave
operator with {\it fixed} size damping on $\Omega$.  Using the
techniques of the previous section we show this implies the same
estimate for the overdamped operator.  

We assume our domain has compact subsets 
\[
V \Subset \tV \Subset \Omega
\]
satisfying 
\[
\tOmega = \overline{ \Omega \setminus \tV }
\]
is a compact subset of
$\reals^n$ and $a$ controls $\Omega$ geometrically outside $V$.  This
implies that $\Omega$ can be extended to an asymptotically Euclidean
scattering manifold, say
\[
X = (\reals^n \setminus U) \cup \tV,
\]
where $U \Subset \reals^n$ and $\partial U = \partial \tV$.  We assume
the semiclassical resolvent with absorbing potential 
\[
Q(h,z) = -h^2 \Delta -z + iW
\]
satisfies polynomial cutoff estimates for energies in a small complex strip $z
\in [1-\alpha, 1 + \alpha] + i(-c_0h, c_0h)$.  That is, if $W \in
\Ci(X)$, $W= 1$ outside a
small neighbourhood of $\tV$ and $W = 0$ on $\tV$ and $\chi \in
\Ci_c(X)$, then we assume
\begin{equation}
\label{E:sc-res-est-W}
\| \chi Q(h,z)^{-1} \chi u \|_{\Hsc^1(X)} \leq Ch^{-1-\delta} \| u
\|_{L^2(X)}
\end{equation}
for some $1 > \delta \geq 0$ and $z \in [1 - \alpha, 1 + \alpha] +
i(-c_0h, c_0h)$.

As in the previous section, we consider $u$ a solution to the overdamped
wave equation \eqref{wave-equation-1} in $\Omega$, for which we have
the following energy decay theorem.

\begin{theorem}
\label{T:wave-decay-2}
Let $u$ be a solution to \eqref{wave-equation-1} with $\Omega$
satisfying the assumptions in \S \ref{S:overdamping-II}, and assume $a \in
\Ci(\Omega)$ controls $\Omega$ geometrically outside $V$.  Then for
every $\epsilon>0$, there exists a constant $C>0$
such that
\begin{equation}
\label{E:energy-decay-2}
\| \partial_t u \|_{L^2(\Omega)}^2 + \| \nabla u \|_{L^2(\Omega)}^2
\leq C e^{-t/C} \| f \|_{H^\epsilon(\Omega)}^2.
\end{equation}
\end{theorem}

\begin{remark}
The assumptions of Theorem \ref{T:wave-decay-2} are satisfied in
several settings.  Extending the example of \cite{CdVP-I} to be Euclidean
outside a compact set satisfies these assumptions, as well as the
cases studied in \cite{Chr-NC,Chr-NC-erratum,Chr-QMNC}.  
\end{remark}

The proof of Theorem \ref{T:wave-decay-2} is very similar in spirit to
the proof of Theorem \ref{T:wave-decay-1}.  We again formally cut off
in time and rescale to get a semiclassical operator as in
\eqref{E:sc-dw-1}.  We have the following estimate on the operator $P(z,h)$.

\begin{proposition}
\label{P:sc-dw-prop-2}
Under the assumptions of Theorem \ref{T:wave-decay-2}, there exist constants $h_0>0$, $\alpha >0$, and $C>0$ such that for $0
< h \leq h_0$ and 
\[
z \in [1 - \alpha, 1 + \alpha] + i(-\infty , h/C],
\]
the operator $P(z,h)$ is invertible as an operator $\Hsc^1(\Omega) \to
L^2(\Omega)$ and
\[
\| P(z,h)^{-1} g \|_{\Hsc^1(\Omega)} \leq \frac{C}{h^{1 + \delta}} \| g
\|_{L^2(\Omega)},
\]
where $0 \leq \delta <1$ is given in \eqref{E:sc-res-est-W}.
\end{proposition}

\begin{proof}
As in the proof of Proposition \ref{P:sc-dw-prop-1}, we assume for
contradiction that there is a sequence of $h_n \to 0$, an $\Hsc^1(\Omega)$ normalized sequence
$u_n$ and a sequence $z_n \in \cx$ such that $\Im z_n = o(h_n)$ and
\[
P(z_n, h_n) u_n = o(h_n^{1+\delta}).
\]
We again decompose $u_n = w_n + v_n$ where $v_n$ is localized to $\{ a
\leq h/C \}$ and $w_n$ is the complement.  Again there are the two
cases of a normalizable subsequence of either the $w_n$ or the $v_n$.
In the case of $w_n$, the argument proceeds exactly as in the proof of
Proposition \ref{P:sc-dw-prop-1}.

Computing the commutators as in \eqref{E:P-chia-comm} and using as in
the estimation of \eqref{E:P-chia-comm} that $\nabla a =\O(h^{(k-1)/k})$, we can take $k$ large enough
so that the right hand side of \eqref{E:sc-dw-3} is $o(h^{1+ \delta})$.
Hence we consider the equation 
\begin{equation}
\label{E:v-g-1}
\tP(z,h) v = \tg
\end{equation}
for $\| v
\|_{\Hsc^1(\Omega ) } = 1$, $\| \tg \|_{L^2(X)} = o(h^{1 + \delta})$,
and 
\[
\tP(z,h) =  (-h \div(1 + i  \sqrt{z} b) h \nabla - z ),
\]
for $b$ which controls $\Omega$ geometrically outside $V$.  

Using the black box framework of \cite{BuZw-bb}, we have the estimate
\begin{align*}
\| u \|_{\Hsc^1 (\Omega )  } & \leq  C h^{-1-\delta} \| \tP(z,h) u
\|_{L^2(\Omega)} + C h^{-\delta} \| b u \|_{L^2(\Omega)} \\
& \leq C h^{-1-\delta} \| \tP(z,h) u \|_{L^2(\Omega)}.
\end{align*}
But then our functions $v$ satisfying \eqref{E:v-g-1} should
satisfy
\[
\| \tP(z,h) v \|_{L^2(\Omega)} \geq h^{1 + \delta }/C \| v
\|_{\Hsc^1(\Omega)} = h^{1 + \delta} /C,
\]
which is a contradiction to the assumption that $\| \tg
\|_{L^2(\Omega)} = o(h^{1 + \delta})$.

This contradiction proves Proposition \ref{P:sc-dw-prop-2} and hence
Theorem \ref{T:wave-decay-2}

\end{proof}


\bibliographystyle{alpha}
\bibliography{DWbib}

\end{document}